\documentclass[reqno,twoside,final]{amsart}
\usepackage{amsmath,amstext,amsthm,amsfonts,amssymb,amscd}
\usepackage[usenames]{color}

\def\Gal{\operatorname{Gal}}

\def\mychar{\operatorname{char}}
\DeclareMathOperator{\SB}{SB}
\def\B{\operatorname{Br}}

\newcommand{\Z}{\mathbb{Z}}

\newcommand\sub{{\,\subseteq\,}}

\newcommand{\Cent}{{\operatorname{Cent}}}
\newcommand{\tr}{{\operatorname{tr}}}

 \DeclareMathOperator{\End}{End}
\def\co{{\,:\,}}
\def\ra{{\rightarrow}}
\newcommand\M[1][n]{{\operatorname{M}_{#1}}}
\newcommand\tensor[1][]{{\otimes_{#1}}}
\newcommand\hra{{\hookrightarrow}}
\newcommand\isom{{\,\cong\,}}
\newcommand\dom{{\backslash}}
\newcommand\dc[3]{{{#1}\dom{#2}/{#3}}}
\newcommand\Ann{{\operatorname{Ann}}}
\newcommand\card[1]{{\left|#1\right|}}
\newcommand\Norm[1][]{\operatorname{N}_{#1}}
\newcommand\hit{{\operatorname{ht}}}
\long\def\forget#1\forgotten{{}}

\theoremstyle{definition}
\newtheorem{definition}{Definition}
\newtheorem{defn}[definition]{Definition}

\newtheorem{cor}[definition]{Corollary}
\newtheorem{rem}[definition]{Remark}
\newtheorem{exmpl}[definition]{Example}
\newtheorem*{rem*}{Remark}
\newtheorem*{acknow*}{Acknowledgements}
\newtheorem*{examples*}{Examples}

\theoremstyle{plain}

\newtheorem{lem}[definition]{Lemma}

\newtheorem{thm}[definition]{Theorem}
\newtheorem{prop}[definition]{Proposition}
\newtheorem*{theorem*}{Theorem}
\newenvironment{proof-sketch}{\noindent{\bf Sketch of Proof}\hspace*{1em}}{\end{proof}\bigskip}
\newenvironment{proof-idea}{\noindent{\bf Proof Idea}\hspace*{1em}}{\end{proof}\bigskip}
\newenvironment{proof-of-lemma}[1]{\noindent{\bf Proof of Lemma #1}\hspace*{1em}}{\end{proof}\bigskip}
\newenvironment{proof-of-prop}[1]{\noindent{\bf Proof of Proposition #1}\hspace*{1em}}{\end{proof}\bigskip}
\newenvironment{proof-of-thm}[1]{\noindent{\bf Proof of Theorem #1}\hspace*{1em}}{\end{proof}\bigskip}
\newenvironment{proof-attempt}{\noindent{\bf Proof Attempt}\hspace*{1em}}{\end{proof}\bigskip}

\def\w{\omega}
\def\a{\alpha}
\def\s{\sigma}

\def\Ker{{\operatorname{Ker}}}
\def\span{{\operatorname{span}}}
\newcommand\set[1]{{\left\{{#1}\right\}}}
\newcommand\sg[1]{{\left<{#1}\right>}}
\newcommand\mul[1]{{{#1}^{\times}}}

\def\F{{\mathbb{F}}}

\newif\iffurther\furtherfalse

\usepackage{xy}
\xyoption{all}

\newcommand\Tref[1]{{Theorem~\ref{#1}}}
\newcommand\Pref[1]{{Proposition~\ref{#1}}}
\newcommand\Lref[1]{{Lemma~\ref{#1}}}
\newcommand\Cref[1]{{Corollary~\ref{#1}}}
\newcommand\Rref[1]{{Remark~\ref{#1}}}
\newcommand\eq[1]{{(\ref{#1})}}

\def\TT{{$(2)$}} 
\def\TR{{$(1)$}} 
\def\TTL{{$(2^*)$}} 
\def\TRL{{$(1^*)$}} 

\makeindex

\begin{document}

\title[bimodule structure of central simple algebras]{Bimodule structure of central simple algebras}

\author{Eliyahu Matzri, Louis H.~Rowen, David J. Saltman, Uzi Vishne}
\address{Department of mathematics, Ben-Gurion University,
Beer Sheva, Israel}
\address{Department of mathematics, Bar Ilan University,
Ramat Gan, Israel}
\address{USA}
\address{Department of mathematics, Bar Ilan University,
Ramat Gan, Israel}

\email {elimatzri@gmail.com }
\email{rowen@math.biu.ac.il}
\email{saltman@idaccr.org}
\email{vishne@math.biu.ac.il}

\thanks{This work was supported by the U.S.-Israel Binational Science
Foundation (grant no. 2010/49)}

\subjclass[2010] {Primary: ; Secondary: }

\date{\today}



\begin{abstract}
For a maximal separable subfield $K$ of a central simple algebra $A$, we provide a semiring isomorphism between $K$-$K$-bimodules $A$ and $H$-$H$ bisets of $G = \Gal(L/F)$, where $F = \operatorname{Z}(A)$, $L$ is the Galois closure of $K/F$, and $H = \Gal(L/K)$. This leads to a combinatorial interpretation of the growth of $\dim_K((KaK)^i)$, for fixed $a \in A$, especially in terms of Kummer sets.
\end{abstract}


\maketitle

\section{Introduction}
$A$ always denotes a simple finite dimensional algebra of degree $n$ (i.e., dimension~$n^2$) with
center $F$, i.e., a central simple algebra $A/F$, and $K =
F[\theta]$ is a given maximal subfield of $A$ that is a separable
extension of $F$. Recall by the Koethe-Noether-Jacobson Theorem
\cite[Theorem~VII.11.1]{J}, \cite{Jac} that, for $A$ a division
algebra, the set of such $\theta\in A$ is Zariski dense in $A$.
The overall goal of this research is to investigate the internal
structure of $A$ in terms of $K \subset A$ and an element $a\in
A\setminus K$.

Notice that assuming $A$ is a division algebra is much too strong for the density statement. It
is clear, for example, that such a field $K$ exists if $F$ has the
property that every finite extensions $F'$ of $F$ has a separable
field extension of any degree. However, we will avoid these
technicalities by simply assuming, in the whole paper, that $K$ is a
maximal separable subfield of $A$.

 We start by reviewing the well-known fact that
there are $v \in A$ such that $A = K v K$ and in fact $A = \sum
_{i,j=0}^{n-1} F\theta^i v \theta^{n-1-i}$ for suitable $v\in K$,
and conversely, starting with any $v$, the set of $\theta$ for which
 $A = K v K$ is Zariski dense in $A$. Likewise, starting with $K$,
 the set of $v$ for which
 $A = K v K$ is Zariski dense in $A$. However, the situation can differ for $KaK$ for arbitrary
$a\in A$, which is the subject of our paper.

We are interested in those
$a\notin K$ for which $KaK \ne A,$ since they may permit us to
obtain more information about $A$. In this case we are also interested
in examining $(KaK)^m$ for each $m\ge 1$. Another focus
of this paper is spaces of Kummer elements. We say $a \in K$
is Kummer if and only if the characteristic polynomial of $a$ has the form $x^n - d$. We say a subspace $V \subset A$ is Kummer if and only
if all $a \in V$ are Kummer.

One extreme case, since $KaK$ must contain $aK$, is that $KaK = aK$.
Then $Ka = aK$. Since $K$ is a field,
by Lemma \ref{invelem}, we can conclude that
$a$ is invertible and $aKa^{-1} = K$.
Some equivalent conditions in terms of traces are given in
Theorem~\ref{Kum1}, which shows that if  $Ka$
is Kummer ($v^n \in F$ for all $v \in Ka$), then $KaK$ is also
Kummer. 
(We utilize characteristic free techniques, developed in Section~\ref{sec:7}, which are of independent interest).

We show that the sequence
$\dim _K(KaK)^j:\ j =1,2,\dots$ must stabilize at some $m \le n$.
Moreover, the sequence $(KaK)^j$ terminates in a finite cycle and
we can characterize the resulting subspaces.

This is tied in with the behavior of the products
$K(aKa^{-1})(a^2Ka^{-2})\cdots$ of fields conjugate to $K$, since
$$KaKa^{-1}(a^2Ka^{-2}) = KaKaKa^{-2} = (KaK)(KaK)a^{-2}$$ (and
also for longer products).

Our first main idea is that this question can be studied
ring-theoretically. $A$ is a $K - K$ bimodule (i.e.
a $K \otimes_F K$ module) and $(KaK)^m$ is a submodule,
where the first $K$ acts by
multiplication on the left and the second as multiplication on the
right. Since we show $A = KrK$ for some $r$
we have that $A \cong K \otimes_F K$ as $K$ - $K$ bimodules.

This sets the tone of our paper, in which we explore first
what one can obtain using the bimodule structure, before studying
how they interact in the multiplication of $A$.

Writing $f$ for the minimal polynomial of $\theta$ over $F$, we
can factor $f$ over $K$ and have $$f(x) = (x - \theta)f_1(x)\cdots f_s(x).$$
Thus, $K \otimes_F K$ is a semisimple ring, a direct sum of fields $K_0 \oplus K_1
\oplus \cdots \oplus K_s$ where $K_0 = K$ and $K_i =
K[x]/K[x]f_i(x)$ for $1 \le i \le s$. As a bimodule $A$ is semisimple, i.e., is
a finite direct sum of simple submodules. This gives us precisely
the description of all sub-bimodules of $A$ (Remark~\ref{doublecorr}).
Of course, the $(KaK)^m$ for each $m$ are such sub-bimodules.

Furthermore, we can describe these $K_i$ in terms of the Galois
group $G$ of the Galois closure $L/F$, generated over $F$ by the
roots of $f(x)$. $G$ acts on these roots, enabling us to translate
the theory to double cosets of $G$. Let $H \subset G$ be the Galois
group of $L/K$. We call a subset $S \sub G$ an {\bf{$H$-biset}}, if
it is closed under multiplication by $H$ from left and right. Then
(Corollary~\ref{doublecorr}) there is a 1:1 correspondence between
sub-bimodules of $K \otimes_F K$ and $H$-bisets $S \sub G$, which is
extended to an isomorphism of semirings in Theorem~\ref{main}. This
valuable tool enables us to relate the algebraic structure of $A$ to
the growth of the series $\dim _K (KaK)^j$.

To relate these sub-bimodules to the multiplication in~$A$,
we embed $A$ into
$\bar A = A \otimes _F L \cong \M[n](L)$. Then $\M[n](L)$
has an etale maximal subring $\bar K = K \otimes_F L$.
After tensoring by $L$, sub-bimodules of $A$ become $\bar K$--$\bar K$
sub-bimodules of $\M[n](L)$. These later sub-bimodules are described in terms
of matrix units and this is a powerful tool. 

Another approach to studying $A$ in terms of $KaK$ is by means of
Brauer factor sets, cf.~\cite{Jac3}, and the corresponding
description of $A$ as matrices of $\M[n](\bar K)$. Now $G$ acts
naturally on the indices of the entries of the matrices, and $K$
corresponds to diagonal matrices. When $KaK \ne A$ this matrix
description involves entries which are 0.

\section{Writing $A = KaK.$}

The fact that we can write $A = KaK$ is known, cf.~\cite{A},
\cite[Theorem~VII.3]{J}, \cite{J1}, and~\cite{G}. Let us provide the
quick argument.

We recall from \cite[Theorem 1.4.34]{Row1} that over any field $F$,
the Capelli polynomial
$$c_{n^2}(x_1, \dots, x_{n^2},y_1, \dots, y_{n^2})$$ has the
property of vanishing whenever $x_1, \dots, x_{n^2}$ are specialized
to sets of matrices that do not span $\M[n](F),$ but does not vanish
when
 $x_1, \dots, x_{n^2}$ and $y_1, \dots, y_{n^2}$ each are specialized to sets of matrices that do span $\M[n](F).$ Define $\tilde x_{ni+j+1} = y^i x y^j$ for $0 \le i,j \le n-1$, and
$$\tilde c(x,y) = c_{n^2}(\tilde x_1, \dots, \tilde x_{n^2},\tilde x_1, \dots, \tilde x_{n^2}).$$

\begin{lem}\label{polysw}
$A$ is spanned by $S = \{\theta^i a \theta^j: 0 \le i,j \le
n-1\}$ iff $\tilde{c}(a,\theta) \neq 0$.
\end{lem}
\begin{proof}
An immediate consequence of the previous paragraph.
\end{proof}

\begin{rem}\label{passup}
Lemma~\ref{polysw} transfers the condition of generation over the
center to a criterion about polynomial identities (or generalized
identities), which for algebras over infinite fields is known to
pass to tensor extensions (cf.~\cite[Corollary 2.3.32]{Row1}), and
being a non-identity, is a Zariski open condition on $a$ when we fix
$\theta$. Namely, if the $\theta^i a \theta^j$ generate $L
\otimes A$ for some field $L$, then the generalized polynomial
$\tilde c(x,\theta) $ is not a generalized identity of $L \otimes
A,$ and thus not a generalized identity of $A$, implying that the $\theta^i
 a \theta^j$ generate $A$ for a Zariski open subset of $A$.
 \end{rem}

\begin{prop}\label{A=KaK}
Let $K/F$ be a separable maximal subfield of $A$. Let $\theta$ be a
generator of $K/F$. Then there are elements $a \in A$ such that $A$
is spanned by the elements $\theta^i a \theta^j$, $i,j =
0,\dots,n-1$, and for $F$ infinite, the set of such elements $a \in
A$ is Zariski dense in $A$.
\end{prop}
\begin{proof}
Since $K/F$ is separable, we can write $K = F[\theta]$ where
$\theta$ is a separable element. First suppose that $F$ is infinite.
 Passing to $K \tensor A$, we may
diagonalize $1 \otimes b$, and thus identify $1 \otimes K$ with
 diagonal matrices. Hence $K \otimes K$ is all of the diagonal. For any $a \in A$,
 we write $1 \otimes a$ as the matrix $\sum _{i,j} a_{ij}e_{ij}$.
 Then $(K \otimes K)(1 \otimes a)(K \otimes K)$
 is spanned by the matrices
$a_{ij}e_{ij}$, so if $1 \otimes a$ has no zero entries we obtain
$(K \otimes K)(1 \otimes a)(K \otimes K)= K \tensor A$.

This shows that the set of $a$ for which $(K\tensor K)a(K\tensor K) = K\tensor A$ is Zarisky dense. By Remark~\ref{passup}, the set of suitable $a$ is Zariski dense in $A$. For any such $a,$
$$\dim _F (KaK) = \dim _K \left((K \otimes K)(1 \otimes a)(K \otimes
K)\right) = n^2,$$ implying $KaK = A, $ as desired.

If $F$ is finite, then $A$ already is a matrix algebra, and $K/F$
is necessarily cyclic. Hence, the matrix algebra $A$ can be
presented as a cyclic algebra $\sum Kz^i$ where $zKz^{-1} = K$.
Considering $A$ as a subalgebra of $\M[n](K)$, and conjugating so
that $K$ is the diagonal, every element of the form $\sum \alpha_i
z^i$ with nonzero coefficients has non-zero entries as a matrix.
\end{proof}

\section{Bimodule decomposition of $KaK$}

We fix a central simple algebra $A$ and a separable maximal subfield $K$.

\subsection{$A$ as a bimodule}$ $

The algebra $A$ has the structure of a $K$--$K$ bimodule, where
the first $K$ acts by multiplication on the left and the second as
multiplication on the right. Of course this is the same as saying
that $A$ is a module over $K \otimes_F K$. Moreover, since by \Pref{A=KaK} there is $a \in A$ such that $A =
KaK$, we conclude:

\begin{lem}\label{L1}
 $A \cong K \otimes_F K$ as $K$--$K$ bimodules.
\end{lem}

Note that since $K/F$ is separable, $K \otimes_F K$ is a commutative
semisimple algebra with all irreducibles appearing with multiplicity
one, and thus can be written uniquely as a direct sum of simple
modules. This implies that:

\begin{lem}
Any two sub-$K$--$K$-bimodules of $A$ that are isomorphic as
bimodules are equal as subsets.
\end{lem}

In other words, the isomorphism type determines the submodule as a
subset. Also note that since $K \otimes_F K$ is semisimple, all
submodules are cyclic. From this we get:

\begin{lem}
The possible $KvK \subseteq A$, ranging over $v \in A$, are exactly the $K$--$K$
submodules of $A$.\end{lem}

Following \Lref{L1}, we study the bimodule decomposition of $K \otimes_F K$. Since $K/F$ is separable,
\begin{equation}\label{basiK}
K \cong F[x]/F[x]f(x)\end{equation} for an irreducible polynomial $f(x)$ over $F$.
The image of $x$ defines a canonical root $\theta \in K$ of $f(x)$,
and we have an irreducible decomposition over $K$, $f(x) = (x - \theta)f_1(x)\cdots f_s(x)$.

Extending scalars in \eq{basiK}, we have $K \otimes_F K \cong K[x]/K[x]f(x)$, where
$\theta \tensor 1 \mapsto \theta$ and $1 \tensor \theta \mapsto x$.
Thus:

\begin{lem}\label{lem4}
$K \otimes_F K$ is a direct sum of fields $K_0 \oplus K_1
\oplus \cdots \oplus K_s$ where $K_0 = K$ and $K_i = K[x]/K[x]f_i(x)$ for $1 \le i
\le s$. Each $K_i$ is an irreducible $K$--$K$ bimodule.
\end{lem}
In other words, the $K_i$ are precisely the irreducible
$K$--$K$ sub-bimodules of $A$.
Passing to $A$, there are $v_0,\dots,v_s$ such that $A = \bigoplus Kv_iK$, and for each $i$, $Kv_iK \isom K_i$ as $K$-$K$-bimodules.

\subsection{Galois structure of $K/F$}$ $

We can think about the $K_i$ of \Lref{lem4} in terms of the Galois group of $K$.
Explicitly, let $L/F$ be the Galois closure of $K/F$, so $L$ is generated
over $F$ by the roots of $f(x)$. Let $\theta_i$ be a
root of $f_i$, so $K_i = K(\theta_i)$. Then the projection $K \otimes_F K \to K_i$ can be defined by
$\theta \otimes 1 \mapsto \theta$ and $1 \otimes \theta \mapsto \theta_i$.

Let $G$ be the Galois group of $L/F$ and $H \subset G$ the
Galois group of $L/K$. That is, viewing $G$ as a permutation
group on the roots of $f(x)$, $H$ is the stabilizer of~$\theta$.
Another way of saying this is that the set of right cosets $G/H =
\{gH \,|\, g \in G\}$ corresponds to the embeddings $K
\hookrightarrow L$, where $gH$ corresponds to the embedding
defined by $\theta \mapsto g(\theta)$. It then follows that the
roots of each $f_i(x)$ (including $f_0(x) = x - \theta$) are orbits
of $H$ with respect to the action of $H \sub G$ on the roots of $f(x)$.

It is useful to get away from relying on a specific choice of
polynomial, which now is easy. Let $\Theta$ denote the set of roots
of $f(x)$ in $L$, which is isomorphic as a $G$-set to $G/H$
(via $g \mapsto g(\theta)$) and thus to the set of embeddings $K
\hra L$ (where $G$ acts via left composition).

The orbits of $H$ on $\Theta$ are the roots of each $f_i$;
the orbits of $H$ on the embeddings are the embeddings of $K$ into $K_i$; and clearly, the orbits of $H$ on $G/H$ are the double
cosets $HgH$, $g \in G$. This gives a correspondence between double cosets $HgH$
and the $K_i$, given by $\theta_i = g(\theta)$, and hence on the simple direct summands of $K \tensor K$. Note that this is really
independent of the choice of $f$. Indeed, $K_i$ is isomorphic to the subfield of $L$ generated by $K$ and
$g(K)$.
Therefore, the subfield $K(g(\theta))$ corresponds to the double coset $HgH$, and we denote this field, up to $K$-isomorphism, as $K_{[g]}$, as writing the double coset in a subscript seems unwise.

We have shown:
\begin{lem}\label{doublecoset}
The simple direct summands of $K \otimes_F K$ are in one to one
correspondence with the double cosets $HgH$ of $H$ in $G$, given by
$HgH \mapsto K_{[g]}$.
\end{lem}

For future use, we generalize this correspondence to all the
sub-bimodules. We call a subset $S \sub G$ an {\bf{$H$-biset}}, if
it is closed under multiplication by $H$ from left and right. In
other words, an $H$-biset is the union of double cosets of $H$.
\begin{cor}\label{doublecorr}
There is a one to one correspondence between sub-bimodules of $K \otimes_F K$ and $H$-bisets $S \sub G$, given by $S \mapsto \sum_{HgH \sub S} K_{[g]}$.
\end{cor}

Lemma \ref{doublecoset} implies, for example, the following
observation (for $s=2$):

\begin{rem}\label{dt}
 $K \otimes_F K = K \oplus K_1$, the sum of two
fields, if and only if $G$ acts doubly transitively on the roots of
$f(x)$.
\end{rem}

So, for example, if $K/F$ has degree $n$ and $G = S_n$ then $K
\otimes_F K = K \oplus K_1$.

Counting degrees in \Lref{doublecoset} we have:

\begin{cor} \label{order}
The dimension of $K_{[g]} \subset K \otimes_F K$
over $K$ is the quotient $|HgH|/|H|$. More generally the dimension
of the bimodule corresponding to any $H$-biset $S$
is~$\card{S}/\card{H}$.
\end{cor}

\begin{proof} The length of the orbits of $H$ on $G/H$ is exactly the
number of roots of the minimal polynomial $f_i$ of $g(\theta)$ over $K$. \end{proof}

The object $K_{[g]}$ is being viewed here in a number of ways, and we need to
describe them and keep them distinct. Of course we began by viewing $K_{[g]}$ as a submodule of
$K \tensor K$. Up to isomorphism, there is a unique
$K\tensor K$-module corresponding to $HgH$. However, being a simple module $K_{[g]}$ is $(K \tensor K)/M$
for a maximal ideal $M$, and then $K_{[g]}$ is a field, which by our description is
isomorphic to $Kg(K) \subset L$.
Equivalently, $K_{[g]} \cong L^{H(g)}$ where $H(g) = H \cap gHg^{-1}$. Note that the field structure
of $K_{[g]}$ does not determine $HgH$. For example, if $K/F$ is Galois
(so $H = (1)$) then all the $K_{[g]}$'s are isomorphic to $K$. More generally, for any $f \in G$,
$f(Kg(K)) \subset L$ is also isomorphic to $K_{[g]}$ and not equal to $Kg(K)$ unless $f$
is in the normalizer of $H(g)$. In other words,
knowing $K_{[g]}$ as a field DOES NOT uniquely define the bimodule
structure. To make $K_{[g]}$ a bimodule we need to define $(k \tensor k') \cdot a$ which amounts
to defining two embeddings $K \to K_{[g]}$ (actions of $K \tensor 1$ and $1 \tensor K)$
and if we, for the moment, identify $K_{[g]}$
with $Kg(K) \subset L$ then these two embeddings are the identity and $g$. It is natural,
whenever we view $K_{[g]}$ as a subfield of $L$, to choose the first embedding always to
be the identity. That is, we only consider subfields $f(Kg(K)) \subset L$ where $f = h \in H$.
When one changes $Kg(K)$ to $h(Kg(K))$, this is equivalent to changing
from $g$ to $hg$ in $HgH$, and in the original description of
$K_{[g]} = K_i = K[x]/(f_i),$ choosing a different root in $L$ of $f_i$. Note that
$h(Kg(K)) \not= Kg(K)$ in general. That is, adding the extra structure of fixing
$K \subset K_{[g]}$ still does not uniquely define $K_{[g]}$ as a subfield of $L$.

There are many bimodule surjections $K \tensor K \to K_{[g]}$, namely, one for
every generator of $K_{[g]}$, though they all have the same kernel.
However, there is a unique such surjection which is
a ring homomorphsm, namely, the one sending $1$ to $1$. We call
this map $\pi_{[g]}: K \tensor K \to K_{[g]}$. This map is hard to work with
because we do not have a fixed instantiation for $K_{[g]}$. Once we fix an embedding
$K_{[g]} \subset L$ (which is the identity on $K$) we have the composition
$K \tensor K \to L$, defined by $\theta \tensor 1 \mapsto \theta$ and $1 \tensor \theta \mapsto g(\theta)$, which depends on the choice of $gH \subset HgH$
(because the embedding $K_{[g]} \to L$ depends on $gH$) and so
we write this composition as $\pi_{gH}$. When we need to make it clear,
we set $K_{gH} = Kg(K)$ to be the specific intermediate subfield of $L/K$ isomorphic to
$K_{[g]}$. Thus it makes sense to write $\pi_{[g]}: K \tensor K \to K_{[g]}$
but $\pi_{gH}: K \tensor K \to K_{gH}$. Since $K \tensor K$ is semisimple,
$\pi_{[g]}$ splits. That is, there is a unique idempotent $e_{[g]} \in K \tensor K$ such that
$K_{[g]} \isom (K\tensor K)e_{[g]}$ as a bimodule;
clearly $\pi_{[g]}(e_{[g]}) = 1$ and $\pi_{[g]}$ restricts to an
isomorphism on $(K \tensor K)e_{[g]}$. The description of
$\pi_{gH}$ can be summarized by the diagram:

\begin{equation}\label{pii}
 \xymatrix{\ K \otimes_F K \ \ar@{->}[d]^{\pi_{g H}} \ar@{^(->}[r]^{\iota \tensor g}
 & \ L \otimes_F L\ \ar@{->}[d]^m \\ K_{gH} \ar@{^(->}[r] & L}
 \end{equation}
where $\iota \co K \ra L$ is the embedding, and $m \co
L \tensor L \ra L$ is the multiplication map.

\section{Multiplication in $A$}

Our goal is to understand how to multiply, in~$A$, the simple
summands of $A$ as a $K$-$K$ bimodule. We approach this by extending
scalars to split $K$ and $A$.

\subsection{Splitting the extension $K/F$}\label{ss:23}
More precisely, we form $\bar K = L \tensor[F] K$ which is naturally
a subalgebra of $\bar A = L \tensor[F] A \cong \M[n](L)$. By definition,
\begin{equation}\label{L2-}
\bar K = L \otimes_F K \,\isom\, L[x]/L[x]f(x) \,\isom\,
\bigoplus_{g{H} \in G/H} L.\end{equation}

Note that $L \tensor[F] (K \tensor[F] K) \cong (L \tensor[F] K) \tensor[L] (L \tensor[F] K) = \bar K \tensor[L] \bar K$. Thus $K$-$K$-sub-bimodules of $A$ become,
after extending scalars to $L$, $\bar K$-$\bar K$-sub-bimodules of $\bar A$.

Rewriting \eq{L2-} we have
\begin{equation}\label{L2}
\bar K = \sum_{g \in G/H} L e_{gH}
\end{equation}
where the $e_{gH}$ are the respective idempotents. Since
\begin{equation}\label{LLL} \Ker(\pi_{g H}) =
\span_{L}\set{g(k) \otimes 1 - 1 \otimes k \,\, | \,\, k \in
K},
\end{equation}
the components can be characterized as

\begin{equation}\label{LL}
Le_{gH} = \set{\sum a_i\tensor b_i
\,:\, \sum g(k)a_i\tensor b_i = \sum a_i \tensor b_i k\ \ \mbox{for all}\ k \in K};
\end{equation}
noting that $g(\theta)^j \tensor 1 - 1 \tensor \theta^j$ is
divisible by $g(\theta) \tensor 1 - 1 \tensor \theta$, we arrive at
the convenient description
\begin{equation}\label{LL+}
Le_{gH} = \set{\sum a_i\tensor b_i
\,:\, \sum g(\theta)a_i\tensor b_i = \sum a_i \tensor b_i \theta}.
\end{equation}

\begin{rem}\label{refined6}
In the spirit of Corollary~{\ref{doublecorr}} but simpler, there is
a correspondence between $\bar K$ submodules of $\bar K$ and unions
of cosets $S \sub G$, given by $S \mapsto \sum_{gH \sub S} L
e_{gH}$.
\end{rem}

Note that, since the idempotent $e_{gH}$ is minimal, equation
\eq{LL} uniquely defines $e_{gH}$ among all idempotents. Unlike the
idempotents $e_{[g]} \in K \tensor K$, $e_{gH}$ varies according to
the representative in the double coset. In fact, as in
\cite[Section~2.3]{Jac}, one easily verifies that for each $g$,
$$e_{gH} = \prod_{g'H \,:\, g'H \neq gH} (g(\theta) \tensor 1 - g'(\theta) \tensor 1)^{-1} \cdot \prod_{g'H \neq gH} (1 \tensor \theta - g'(\theta ) \tensor 1).$$

Letting $G$ act on the $L$ in $\bar K = L \tensor[F] K$, we
immediately observe from equation~(\ref{LL}) that
\begin{equation}\label{leftaction}
e_{g'gH}= (g' \tensor 1)(e_{gH}), \quad
\forall g' \in G.\end{equation}

\smallskip

It will be useful to view $\bar{K} = L \otimes_F K \subset L \otimes_F L$.
Now, $L \otimes_F L = \oplus_{g \in G} Le_g$ for idempotents $e_g$, and since
$gH$ are exactly the elements of $G$ that agree with $g$ on~$K$,
\begin{equation}\label{dec4}
e_{gH} = \sum_{h \in H} e_{gh}
\end{equation}
as an element of $L\otimes_F L$. Clearly, $(g' \tensor
1)(e_g) = e_{g'g}$. Moreover, from the equation $(g(t) \otimes 1 - 1
\otimes t)e_g = 0$ for all $t \in \bar K$, it also follows that
\begin{equation}\label{act}
(1 \otimes g')(e_g) = e_{gg'^{-1}}.
\end{equation}

We now have three layers of idempotents: $e_{[g]} \in K \tensor K$, $e_{gH} \in L \tensor K$ and $e_g \in L \tensor L$, where $e_{[g]} = \sum_{g'H \sub HgH} e_{g'H}$ and $e_{gH} = \sum_{g' \in gH} e_{g'}$.

\subsection{An example}

To get an example, suppose that $K/F$ contains an intermediate field
$K'$. Let $H'$ be the Galois group of $L/K'$, a subgroup of $G$
containing~$H$. Of course, $H'$ is a union of double cosets of $H$.
We ask, ``What is the sub-bimodule of $K \tensor K$, i.e., of $A$,
corresponding to $H'$?''

The answer provided by \Cref{doublecorr}, in terms of the components $K_{[g]}$, is unsatisfactory, being non-explicit as a subset of $K \tensor[F] K$ or $A$.
Instead, we will state the right answer and proceed to prove it.

In order to utilize idempotents, we note that the isomorphism $K\tensor K \isom A$ as $K \tensor K$-modules extends
to an isomorphism of $L \tensor[F] K = L \tensor[K] (K \tensor[F] K)$ and $L\tensor[K] A$ as $L\tensor K$-modules.
The advantage is that now we have a concrete decomposition of the module, in terms of idempotents and annihilators. Indeed, let $T_{gH} = \Ker(\pi_{g H})$, which is given in \eq{LLL}.
 Viewing $T_{gH}$ as an ideal of the base ring $L\tensor K$, we have that $L e_{gH} = \Ann(T_{gH})$ by \eq{L2}.
 Now, let $S \sub G$ be a subset closed under multiplication by $H$ from the right; the correspondence of
 \Rref{refined6} takes $S$ to $$\sum_{gH \sub S} Le_{gH} = \Ann\left(\bigcap_{gH \sub S} T_{gH}\right),$$
which is isomorphic (as $L$-$K$-bimodules) to the annihilator of $\bigcap_{gH \sub S} T_{gH}$ in its action on $L\tensor[K] A$.
By \Cref{order}, the dimension of this module over $L$ is $\card{S}/\card{H}$.

Let us now describe the submodule of $A$ associated to the subgroup
$H'$. Let $T'$ be the ideal (of $L \tensor K$) generated by
the elements $k' \tensor 1 - 1 \tensor k'$, ranging over $k' \in
K'$. Clearly, $T' \sub T_{gH}$ for any $g \in H'$, since $g(k') =
k'$ for $g \in H'$. Now, the annihilator of $T'$ in its action on
$L \tensor[K] A$ is composed of the elements commuting with
$K'$, so is equal to $L \tensor[K] \operatorname{C}_A(K')$.
The dimension is $[\operatorname{C}_A(K')\!:\!K] = [K\!:\!K'] =
\card{H'}/\card{H}$, noting that $K$ is a maximal subfield of
$\operatorname{C}_A(K')$. To summarize, the annihilator of
$\bigcap_{gH \sub H'} T_{gH}$ is contained in the annihilator of
$T'$, and they have the same dimension, so they are equal.

By descent from $L \tensor K$ to $K \tensor K$, we have
proved:
\begin{lem}\label{L10} $H'$ corresponds to the $K$-$K$ submodule of $A$ which is
the centralizer, $\operatorname{C}_A(K')$, of $K'$ in $A$.
\end{lem}

More generally, suppose there is an intermediate field $K''/F$ where
$K' \subset K'' \subset K$ and $K''/K'$ is cyclic Galois with Galois
group generated by $\sigma$. Let $N$ be the Galois group of $K''$ in
$L$. We have that $H \subset N \subset H'$, $N$ is normal in
$H'$, and $H'/N$ is generated by $\sigma$. Since $\sigma$ normalizes
$N$, $N\sigma$ is a union of $H$ double cosets. Arguing as above
with $T'' $ generated by $ \set{\sigma(\ell) \tensor 1 - 1 \tensor
\ell \,:\, \ell \in L}$, we have:

\begin{prop}\label{subcyclic}
$N\sigma$ corresponds to the submodule
$$\operatorname{C}_A(K'',\sigma) = \{a \in A \,|\, \sigma(\ell)a = a\ell \mbox{ for all } \ell \in K'' \}.$$
\end{prop}

In particular $\operatorname{C}_A(K'',\sigma)$ is non-zero, and as
we will show (\Lref{invelem}), it contains an invertible element.
Thus Lemma~\ref{L1} is actually a generalization of the
Skolem-Noether Theorem.

\subsection{Splitting $K \tensor K$}$ $

Next, we consider the tensor product
$L \otimes_F (K \otimes_F K) =
\bar K \otimes_{L} \bar K$.
Taking the tensor product of \eq{L2} with itself,
we obtain a direct sum decomposition
\begin{equation}\label{dec5}
L \otimes_F (K \otimes_F K) = \sum_{g'H,g''H \in G/H} L(e_{g'H} \otimes e_{g''H}).
\end{equation}

In Subsection~\ref{ss:23} we observed that $L \tensor[F] K = L
\tensor[K] (K \tensor[F] K)$ decomposes as $\sum L e_{gH}$. The
action of $H$ on $L$ translates to the action of $H$ on the set of
idempotents corresponding to $G/H$, so the orbits correspond to the
double cosets $\dc{H}{G}{H}$. Similarly, the action of $G$ on $L$ in
$L \tensor[F] (K\tensor[F] K)$ translates to the natural action on
the pairs $e_{g'H}\otimes e_{g''H}$, which correspond to $G/H \times
G/H$. We observe that the invariant space is $K \tensor[F] K$ in
both cases, while the actions on idempotents demonstrate the set
isomorphism $G \dom (G/H \times G/H) \isom \dc{H}{G}{H}$ given by
$$G\cdot (xH,yH) \mapsto Hy^{-1}xH.$$

\subsection{Description of $L \tensor K_{[g]}$}$ $

Let $K_{[g]} \subset K \otimes_F K$ be the direct summand
corresponding to the double coset $HgH$ with idempotent $e_{[g]}$.
Thus $e_{[g]}$ is, after tensoring over $F$ by $L$, the sum of
idempotents of the form $e_{g'H} \otimes e_{g''H}$ and we need to
determine which ones appear. We defined a projection $\pi_{gH} \co K
\otimes_F K \to K_{gH} \cong K_{[g]}$, and we also use $\pi_{g H}$ to
denote its $L$-linear extension $L \otimes_F (K \otimes_F
K) \to L \otimes_F K_{gH}$. We have an induced morphism (also called $\pi_{gH}$):

$$\pi_{gH}: \bar K \tensor[L] \bar K = L \tensor[F] (K \tensor[F] K) \to L \tensor[F] K_{gH}
\subset L \tensor[F] L.$$
Since we are going to apply this $\pi_{gH}$ to idempotents of the form $e_{g'H} \otimes e_{g''H}$
let us record the precise definition, via a commutative diagram
built on \eq{pii}, where $m$ denotes the multiplication of $L$ in the right-most vertical arrow and the multiplication of $L \tensor L$ in the right-most diagonal arrow. All the undecorated tensor products are over $F$.
\begin{equation}\label{piiii}
\xymatrix@C=6pt{
{} & (\bar K) \tensor[L] (\bar K) \ar@{->}[rr]^{(1\tensor \iota) \tensor (1 \tensor g)} \ar@{->}[ddl]|(0.5)\hole^(0.65){\pi_{gH}} & {} & (L \tensor L) \tensor[L] (L \tensor L) \ar@{->}[ddl]^{m} \\
L \tensor (K \tensor K) \ar@{->}[d]_{\pi_{gH}} \ar@{->}[rr]^{1 \tensor (\iota \tensor g)} \ar@{->}[ru]|{\cong} & {} & L \tensor (L \tensor L) \ar@{->}[d]_{1 \tensor m} \ar@{->}[ru]|{\cong} & {} \\
L \tensor K_{gH} \ar@{^(->}[rr] & {} & L \tensor L & {}
}\end{equation}

Let $H(g) = H \cap gHg^{-1}$ which we saw was the
Galois group of $L$ over $K_{gH}$. Exactly as in \eq{L2-} and \eq{L2},
$$L \otimes K_{gH} =
\bigoplus_{fH(g) \in G/H(g)} Le'_{fH(g)} \isom \sum_{fH(g) \in G/H(g)} L $$
for idempotents $e'_{fH(g)}$.
(The $e'_{fH(g)}$ are components of the $e_{fH} \in L \tensor
K$, in the sense that  $e_{fH} = \sum_{f'H(g) \sub fH}
e'_{f'H(g)}$, in analogy to \eq{dec4}.)

\begin{prop}
Let $g',g'' \in G$. The image $\pi_{gH}(e_{g'H} \otimes e_{g''H})$
is the primitive idempotent $e_{fH(g)}',$ if $f \in G$ restricts to
$g'$ on $K$ and to $g''g^{-1}$ on $g(K)$. If there is no such $f$,
then $\pi_{gH}(e_{g'H} \otimes e_{g''H}) = 0$.
\end{prop}

\begin{proof}
We apply $\pi_{gH}$ to $e_{g'H} \tensor e_{g''H}$ using diagram \eq{piiii}, taking the route to the right and then down. The first step takes us to $e_{g'H} \tensor (1\tensor g)(e_{g''H})$, where each entry is now an element of $\bar{K} \tensor \bar{K}$, where we can apply \eq{dec4} to get the sum
$$\sum_{h',h'' \in H} e_{g'h'} \tensor (1\tensor g)(e_{g''h''}).$$
Applying \eq{act}, this is equal to
$$\sum_{h',h'' \in H} e_{g'h'} \tensor e_{g''h''g^{-1}}.$$
Multiplication in $L \tensor L$ takes us now to
$$\sum_{h',h'' \in H} e_{g'h'} e_{g''h''g^{-1}} = \sum_{h',h'' \in H} \delta_{g'h',g''h''g^{-1}}e_{g'h'},$$
where $\delta$ is the Kronecker delta. Thus, the image is nonzero
iff there are $h',h'' \in H$ such that $g'h' = g''h''g^{-1}$.

Assume that $f = g'h' = g''h''g^{-1} = g''g^{-1}(gh''g^{-1})$. Note
that $f \in g'H$ is equivalent to $f$ restricting to $g'$ on $K$;
and $f \in g''g^{-1}(gHg^{-1})$ is equivalent to $f$ restricting to
$g''g^{-1}$ on $g(K)$.

If this happens for one pair $g'h'$ and $g''g^{-1}gh''g^{-1}$ the
same is true for $g'h'h$ and $g''g^{-1}gh''g^{-1}h$ for any $h \in H
\cap gHg^{-1} = H(g)$, and so when some such pair exists we have
$\pi_{[g]}(e_{g'H} \otimes e_{g''H}) = e_{fH(g)}$ as needed.
\end{proof}

\begin{cor} \label{idempotents}
The idempotent $e_{g'H} \otimes e_{g''H}$ appears in $L \otimes_F K_{gH}$
if and only if $g'^{-1}g'' \in HgH$. In other words, as a subalgebra of \eq{dec5},
$$L \tensor K_{[g]} = \sum_{g'^{-1}g'' \in HgH} L (e_{g'H} \otimes e_{g''H}).$$
\end{cor}
\begin{proof}
We proved in the proposition that the idempotent appears in $L \tensor K_{gH}$ iff
$g'H \cap g''Hg^{-1} = g'H \cap g''g^{-1}(gHg^{-1})$ is nonempty. This is equivalent to
the second statement because if $g'h' = g''h''g^{-1}$ then $g'^{-1}g'' = h'gh''^{-1}$.
\end{proof}

\begin{rem}
Note that the decomposition of $A$ into submodules isomorphic
to the $K_{[g]}$'s (associated with double cosets) is a
generalized grading of $A$, and has as a special case the known
gradings when $H = 1$. When combined with Theorem \ref{main}  to
come,  the $K$-$K$ bimodule decomposition of $A$ will be seen to be
a generalized grading as well.
\end{rem}

\subsection{The matrix representation of $A$}$ $

Our next step is to understand the product of sub-bimodules of $A$,
which are all of the form $KaK$ for $a\in A$, as subsets of $A$. One
might think we have to specify and understand $A$. However, the fact
that isomorphic $K$-$K$ sub-bimodules are equal implies we can
multiply the sets inside $L \tensor A = \M[n](L)$, and the Brauer
class of $A$ does not matter.

As we stated above, $L \tensor A$ is the matrix algebra
$\M[n](L)$, but we want to be more specific. Since $L
\otimes_F K$ is a direct sum of copies of $L$, we may assume
that
\begin{equation}\label{triv-e}
\bar K = L \otimes_F K \subset \bar A
\end{equation}
are diagonal matrices. More specifically, the idempotents $e_{gH}$
of \eq{L2} are then diagonal idempotents. Moreover, we can choose
matrix units $e_{gH,g'H}$ for $\bar A$, such that the embedding
\eq{triv-e} sends $e_{gH}$ to the diagonal matrix unit $e_{gH,gH}$.
In addition, we can be very free to choose $v \in \bar A$ such that
$(\bar K)v(\bar K) = \bar A$, subject only
to the condition that all the matrix entries of $v$ are nonzero.
Furthermore, taking $v$ to be the all-$1$ matrix, the bimiodule
isomorphism $(\bar K) \otimes_{L} (\bar K) \ra \bar A$ defined by $x \tensor y \mapsto xvy$, maps $e_{gH}
\otimes e_{g'H}$ to $e_{gH,g'H}$.

Recall that our starting point is the isomorphism of
$K$-$K$-bimodules $K \tensor K \ra A$, as in the bottom row of
Figure~\ref{fig1} (see below). However, since we adjust the
isomorphism on the upper row to send $e_{gH} \otimes e_{g'H}$ to
$e_{gH,g'H}$, the diagram does \emph{not} commute when the side
arrows are the natural embeddings (since the all-$1$ matrix is not
an element of $A$). Our strategy still works, because a submodule of
$K \tensor K$ has the same image in $\bar{K} \tensor A$ under both
routes of the diagram, since the images are isomorphic as bimodules.

Following \Cref{doublecorr}, we thus have a one to one
correspondence between $H$-bisets $S \sub G$ and
$K$-$K$-sub-bimodules of $A$, given by
\begin{equation}\label{Phidef}
\Phi \co S \mapsto \sum_{HgH \sub S} \varphi_v(K_{[g]}),
\end{equation}
where $\varphi_v(x\tensor y) = xvy$ and $v \in
A$ is any element for which $KvK = A$.

\begin{figure}[!h]
$$
\xymatrix{\bar{K} \tensor (K \tensor K) \ar@{->}[r] & \bar{K} \tensor A \\
K \tensor K \ar@{^(->}[u] \ar@{..>}[r]& A \ar@{^(->}[u]}$$
\caption{The diagram does not commute on elements when the upper
arrow sends $e_{gH} \otimes e_{g'H} \mapsto e_{gH,g'H}$, but does
commute on sub-bimodules}\label{fig1}
\end{figure}

From Corollary \ref{idempotents} we have:

\begin{prop} \label{matrixform}
With $\bar A$ and $e_{gH,g'H}$ as above, if $KaK \subset A$ is the
simple bimodule associated to the double coset $HgH$, then
$L\otimes_F (KaK) \subset \bar A$ is spanned by all $e_{g'H,g''H}$
such that $g'^{-1}g'' \in HgH$.\end{prop}

From this we get our main theorem on products of bimodules. Notice that the set of $H$-bisets in $G$ is a semiring with respect to the operations of union and multiplication (with the empty set as a zero element and $H$ as a multiplicative unit).

Similarly, the sum of bimodules and the product of bimodules in $A$ are bimodules, and moreover the product is distributive with respect to the sum. The induces a semi-ring structure on the $K$-$K$ sub-bimodules of $A$ (where the zero module is a zero element and $K$ is a multiplicative unit).

\begin{thm} \label{main}
Let $A/F$ be a central simple algebra with maximal separable
subfield~$K$.
The map $\Phi$ of \eq{Phidef} is an isomorphism of semirings, from the semiring of $H$-bisets in $G$
to the semiring of $K$-$K$ sub-bimodules of $A$.
\end{thm}
\begin{proof}
Suppose that $KaK$, $Ka'K$ are $K$-$K$ sub-bimodules associated to
the $H$-bisets $S,S' \sub G$, respectively. Because isomorphism and
equality of sub-bimodules are equivalent, it suffices to prove this
result after extension of scalars to $\bar A = \bar K \otimes_F A$.
Then $(\bar K \otimes_F (KaK))(\bar K \otimes_F (Ka'K))$ is spanned
by all products $e_{g_1H,g_2H}e_{g_2H,g_3H} = e_{g_1H,g_3H}$ where
$g_1^{-1}g_2 \in S$ and $g_2^{-1}g_3 \in S'$. But
$(g_1^{-1}g_2)(g_2^{-1}g_3) = g_1^{-1}g_3$ is the general element of
$SS'$.
\end{proof}
In fact our semirings are idempotent semirings ($x+x=x$ for every $x$), and they come with some extra structure. The additive atoms are double cosets on one hand, and irreducible sub-bimodule on the other hand. There is also an involution, defined by inversion on bisets, and by $\varphi_v(K_{[g]}) \mapsto \varphi_v(K_{[g^{-1}]})$ on irreducible sub-bimodules.

\begin{cor}\label{main-sg}
Let $S$ be an $H$-biset in $G$. Then $\Phi(S)$ is a subalgebra if and only if $S$ is a subgroup.
\end{cor}

It will be useful to observe a property of $KaK$ relating to
the trace.

\begin{lem}
Let $KaK \sub A$ be a sub-bimodule corresponding to a biset $S$. If $K \not\subseteq KaK$ then $\tr(KaK) = 0$.
\end{lem}

\begin{proof}
Let $S$ be the biset corresponding to $KaK$.
After passing to $\bar K \otimes_F A$,
we note by Proposition~\ref{matrixform} that matrix units from the principal diagonal
are in $L \otimes_F KaK$ iff $H \sub S$, iff $K = \Phi(H) \sub \Phi(S) = KaK$.
\end{proof}

\section{Powers of indecomposable modules}

In this section, we consider the series $(KaK)^m$ as $m$ increases.
By \Tref{main}, we have:
\begin{cor}\label{main-powers}
Let $KaK$ be the sub-bimodule associated to the $H$-biset $S \sub
G$. Then $(KaK)^m$ is the bimodule associated to $S^m =
\set{s_1\dots s_m \,:\, s_1,\dots,s_m \in S}$.
\end{cor}

Before going any further, let us settle a technical point.

\begin{lem}\label{invelem}
Suppose that $F$ is a field of order $> 2$, and $A$, $K$ are as
usual. If $KaK$ is a sub-bimodule, we can choose $a$ to be
invertible.
\end{lem}

\begin{proof} We begin with the case $F$ infinite. Associated to $KaK$ is a set of idempotents
$e_{gH,g'H} \in \bar K$. Form the generic element $T = \sum
x_{gH,g'H}e_{gH,g'H}$ where the coefficients are indeterminates.
Viewing $T$ as a matrix, let $t$ be the determinant. Then $t$ is a
polynomial in the $x_{gH,g'H}$ whose coefficients are $\pm 1$. Since
the $F$ points of $KaK$ are Zariski dense, it suffices to show that
$t$ is nonzero. However, there is a field extension $F' \supset F$,
and a division algebra $D'$ over $F'$ such that $K' = K\otimes_F F'$
is a maximal subfield of $D'$ and $\bar K' = \bar K \otimes_F F'$ is
a field, with $G,H$  the corresponding Galois groups of $K'$. Now
$D'$ has a $K'$-$K'$ sub-bimodule corresponding to the same set of
double cosets, which has the form~$K'dK'$ where $d'$ is obviously
invertible. But this shows that $t$ is nonzero.

Next suppose $F$ has finite order $q > 2$. Then $K/F$ must be cyclic Galois
and $A$ is the cyclic algebra $(K/F,\sigma,1)$ where $\sigma: K \to K$
is defined by $\sigma(k) = k^q$. Write $A = \sum_{i=0}^n Ku^i$ where $uk = k^qu$
and $u^n = 1$. Viewing $A = \End_F(K)$, then $k \in K$ acts by left
multiplication and $u(x) = x^q$. Then any bimodule is of the form $\sum_{i \in I} Ku^i$ where
$I \subset \{1,\ldots,n-1\}$. It suffices to show that for any such $I$
there are nonzero $k_i$ such that $\sum_{i \in I} k_iu^i$ is nonsingular.

We proceed by induction on the order of $I$. Multiplying by a power of $u$,
we can always assume that $0 \in I$. If $|I| = 1$, then $1$ is nonsingular.
The induction step is covered by the following lemma.

\begin{lem}
Assume that $\card{F} > 2$. Suppose that $B \co K \to K$ is
$F$-linear and nonsingular. Then there is an element $k \in K^*$
such that $k + B$ is nonsingular.
\end{lem}
\begin{proof}
If $F$ is infinite one can take $k \in F$ since there are finitely
many eigenvalues. Assume that $F$ is finite. If $(k + B)x = 0$ for
nonzero $x \in K$ then $k = -B(x)/x$. It suffices to show that $x
\mapsto B(x)/x$ as a function $K^* \to K^*$ is not surjective. But
if $a \in F^*$ then $B(ax)/(ax) = B(x)/x$ so when $\card{F} > 2$,
this map is not injective.
\end{proof}
\end{proof}
\begin{rem} Note that the above result is false if $F$ has order 2. In the notation
of the above proof, $u(x) = x^2$. Thus $k + u$ is always singular because $kx + x^2 = (k + x)x$
and so for any $k$, $x = -k$ is a kernel element.\end{rem}

From now on, for simplicity, we assume that $KaK$ is a simple sub-bimodule
associated via $\Phi$ to the single double coset~$S = HgH$.
We also assume throughout that $a$ is invertible. Define
$$K(a,m) = (KaK)^m a^{-m} = K(aKa^{-1})(a^2Ka^{-2})\cdots(a^{m}Ka^{-m}).$$

\begin{lem}\label{stab00} If $K(a,m) = K(a,m+1)$ then $K(a,m) = K(a,m+s)$ for all $s \geq 0$. Moreover,
\begin{equation}\label{SerK}
K = K(a,0) \sub K(a,1) \sub K(a,2) \sub \cdots
\end{equation} stabilizes at some $m \le n$. Thus, $\dim_K(KaK)^m$ stabilizes at the same $m$.

Similarly
\begin{equation*}
H(g,m) = (HgH)^m g^{-m} = H(gHg^{-1})(g^2Hg^{-2})\cdots(g^mHg^{-m}),\end{equation*}
is an ascending chain of subsets
\begin{equation}\label{SerH}
H = H(g,0) \sub H(g,1) \sub H(g,2) \sub \cdots,
\end{equation}
so $\card{(HgH)^m}$ stabilizes.
\end{lem}
\begin{proof} To prove the first statement, assume that
$K(a,m) = K(a,m+1)$. It suffices to show that $K(a,m+1) = K(a,m+2)$.
But $(KaK)^ma^{-m} = (KaK)^{m+1}a^{-(m+1)}$ implies $(KaK)^ma =
(KaK)^{m+1}$ and so
\begin{equation}\begin{aligned}(KaK)^{m+1}a^{-(m+1)} & =
(KaK)(KaK)^ma(a^{-(m+2)})
 \\ &= (KaK)(KaK)^{m+1}a^{-(m+2)} =
K(a,m+2).\end{aligned}\end{equation}

As for the second statement, since $\dim_K A = n$ this ascending series
must repeat for $m \leq n$ and the result follows.
\end{proof}

\begin{defn} The {\bf{height}} of $a$ is
the minimal $m_0$ such that $$K(a,m_0) = K(a,m_0+1);$$ we denote
$\hit(a) = m_0$.\end{defn}

 Notice that the height $m_0$ only depends on
the bimodule $KaK$ and we are really talking about the height
of $KaK$. Similarly, we can talk about the height of $g$ or $HgH$.
Of course, if $KaK$ is associated to $HgH$ then they have the same
height.

 The obvious question concerns the possible asymptote for
the sequence $$H(g,m) = (HgH)^{m} g^{-m} =
H(g{H}g^{-1})(g^2Hg^{-2})\cdots({g}^mHg^{-m}).$$ When this
sequence stabilizes, it has the following properties:

\begin{lem}\label{powers1}
Let $m_0$ be the height of $a$.
Let $N = H(g,m_0)$ and let $G'$ be the subgroup of $G$ generated
by $H$ and $g$.

Then $N$ is a normal subgroup of $G'$, $G'/N$ is generated by the image of $g$, and $N$ is the smallest
subgroup of $G'$ containing $H$ and normal in~$G'$.
\end{lem}
\begin{proof} Let $m_1 \geq m_0$ be such that $g^{m_1} = 1$.
Then $$N = H(g,m_1) = (HgH)^{m_1}g^{-m_1} = (HgH)^{m_1},$$ so
$$N^2 = (HgH)^{2m_1} = (HgH)^{2m_1}g^{-2m_1} = H(g,2m_1) = N,$$
proving that $N$ is a subgroup. As such, it is obviously the
subgroup generated by all the conjugates $g^iHg^{-i}$ and the rest is
clear.
\end{proof}

Let $N$ and $G'$ be as in the above lemma. Recall that $L$ is the Galois closure of $K/F$, $G = \Gal(L/F)$ and $H = \Gal(L/K)$. The groups $$G \supseteq G' \supseteq N \supseteq H \supseteq 1$$ define fields $$F \subseteq K' \subseteq E \subseteq K \subseteq L,$$
where $G'$ is the Galois group of $L/K'$ and $N$ is the Galois
group of $L/E$. Clearly $E/K'$ is a cyclic Galois extension.
In fact $E = \bigcap_i g^i(K)$ by the
Galois correspondence, so $E$ is the maximal subfield of $K$ that is stable under~$g$.

\begin{lem}\label{stab0}
 $(KaK)^m$ is stable under conjugation by $a$ if and only if $m \geq m_0$.
\end{lem}
\begin{proof}
Indeed, $K(a,m+1) = K(a,m)$ iff $(KaK)^{m+1}a^{-1} = (KaK)^m$ by
multiplication by $a^m$ from the right, but the left-hand side is
$Ka(KaK)^{m}a^{-1}$, so we have an equality iff $a(KaK)^{m}a^{-1}
\sub (KaK)^m$.
\end{proof}

But $(KaK)^{m_0}$ need not be a subalgebra. Accordingly, we must
go a bit further. Let $m_1 \geq m_0$ be such that $g^{m_1} = 1$.

\begin{lem}\label{stab1}
$(KaK)^{m_1}$ is a
subalgebra of $A$.
\end{lem}
\begin{proof}
As in \Lref{powers1}, let $N = (HgH)^{m_1}$ which is a subgroup of
$G$. By \Cref{main-sg}, $\Phi(N) = (KaK)^{m_1}$ is a subalgebra.
\end{proof}

Let $m \geq \hit(a)$. Then $(HgH)^{m} = N g^{m}$. By
\Lref{subcyclic}, we obtain:
\begin{prop} \label{mainseries}
Let $KaK \subset A$ be a simple sub-bimodule
corresponding to $HgH$.
Then for $m \geq \hit(a)$, $(KaK)^m = \{ x \in A \,|\,
x\ell = g^m(\ell)x \mbox{ for all } \ell \in L \}$.
\end{prop}

Notice that the double coset $HgH$ determines $G' = \sg{H,HgH}$ and therefore determines $N$ (as the minimal normal subgroup of $G'$ containing $H$) and $L = K^N$. The proposition shows that $HgH$ provides explicit information on $a$:

\begin{cor}\label{aconj} Let $m_1$ and $g \in G$ be as in Lemma \ref{stab1}.
Let $a \in A$ be any element such that $KaK$ is the simple bimodule corresponding to $HgH$. Then \begin{enumerate}
\item Let $C = (KaK)^{m_1}$. Then $C = \operatorname{C}_A(L)$.
\item $a \ell = g(\ell) a$ for every $\ell \in L$.
\end{enumerate}
\end{cor}
\begin{proof}
By \Pref{mainseries}, $(KaK)^{m_1} = \operatorname{C}_A(L)$. Taking $m_1+1$ gives
$$a \in \operatorname{C}_A(L) a \sub \operatorname{C}_A(L)KaK = (KaK)^{m_1+1} = \set{ x \in A \,|\, x\ell = g(\ell) x \mbox{ for all } \ell \in L}.$$
\end{proof}

\begin{cor}
 $K' = \Cent(A'),$ where $A'$ is the subalgebra of $A$ generated
by $K$ and $a$.
\end{cor}

\begin{proof} Since $C \sub A'$, the centralizer of $A'$ is contained in
$\operatorname{C}_A(C) = L$, so $\operatorname{C}_A(A')$ is the
subfield of $L$ fixed by conjugation by $a$;
$$\operatorname{C}_A(A')
= L^{\sg{g}} = \bar{K}^{\sg{H,g}} = \bar{K}^{G'}.$$
\end{proof}

In the special case where $g$ normalizes $H$, we have that $N = H$ so $L = K$, and in particular $g$ is an automorphism of $K$. In this case, the condition $a \ell = g(\ell) a$ (for all $\ell \in K$) implies that the coset $Ka = KaK$ is well defined by $g$.

Note that the property that $E/K'$ was cyclic arose from the
assumption that $KaK$ was simple. One could obviously
formulate a more general result for more general $KaK$.

\section {Examples for $A$ cyclic}

We consider situations when the algebra $A$ is cyclic,
although the subfield $K$ need not be cyclic.

\subsection {K cyclic}$ $

 Assume that $F$ contains a primitive
 $n$-th root $\w$ of 1.
Throughout, $(K, \s, \beta)$ denotes the cyclic algebra with
maximal subfield $K$ cyclic over $F$ with Galois group~$\langle \s
\rangle $ and element $y$ such that $y ^n = \beta$ and $y a y^{-1} =
\s(a)$ for all $a \in K$. In particular, when $K = F[x]$ with $x^n =
\a$, we write $(K, \s, \beta)$ as the \textbf{symbol
algebra}~$(\a, \beta).$

\begin{prop}\label{mult1} Suppose that $A = (K/F,\sigma,b)$ is a cyclic algebra,
$yk = \sigma(k)y$ for all $k \in K$, and $a \in A$.
Then any subalgebra of the form $KaK$ must have the form $K[y^d]$ for some $d \,|\, n$.
\end{prop}

\begin{proof} Here $H$ is trivial so sets of double cosets are just sets
of elements. We identify $S$ with this subset of $\Z/n\Z$ (the Galois group).
However, by Theorem \ref{main}, $S$ is closed under addition.\end{proof}

\subsection{K cyclic after adjoining roots of unity}{\ }

Let $E/F$ be a Galois extension of dimension $4$, with $\Gal(E/F) =
\set{1,\eta,\eta^2,\eta^3}$. A cyclic algebra of degree $4$ over $F$
which is split by $E$ has the form $A = E[\theta]$ where conjugation
by $\theta$ induces $\eta$, and $\theta^4 \in \mul{F}$. For a
maximal subfield which is not Galois over $F$, we take $K =
F[\theta]$, and assume that $i = \sqrt{-1} \not \in F$. Our goal is
to decompose $A$ as a $K$-$K$ bimodule. The Galois closure of $K$ is $L = K[i]$,
with $G = \Gal(L/F) = \sg{\s,\tau}$, where $\s \co \theta \to
i\theta,\, i \mapsto i$;\, $\tau \co \theta \mapsto \theta,\, i
\mapsto -i$. We calculate that $\tau\sigma = \sigma^{-1}\tau$
so $G$ is dihedral. The Galois group of $L$ over $K$ is $H =
\sg{\tau}$. Also, $E' =E[i]$ is cyclic over $F' = F[i]$, so there is
a Kummer generator $u \in E'$ for which $\theta u \theta^{-1} = i
u$. In particular $\theta u^2 \theta^{-1} = -u^2$ and $u^2 \in
E^{\eta^2}$.

As before, we extend scalars to obtain explicit idempotents.
Extending the bimodule isomorphism $K \tensor[F] K \isom A$, we
have $L \tensor[K] (K \tensor K) \isom L \tensor[K] A$.
In general this would not have a relevant structure as an algebra.
However, taking $F' = F[i]$, we have that $L = F' \tensor[F]
K$, and therefore $L \tensor[K] A = (F' \tensor[F] K)
\tensor[K] A = F' \tensor[F] A$, which conveniently is an algebra.
The component of $L \tensor[K] A$ corresponding to the coset
$g H \in G/H$ is defined in \eq{LL+}, which can be rewritten as
$\set{ \alpha \,:\, g(\theta)\alpha = \alpha \theta}$. There are
four cosets. Over $F'$, the component corresponding to the coset
$\sigma^jH$ is $\set{\alpha \in F'A \,:\, \sigma^j(\theta)\alpha =
\alpha \theta} = \bar{K} u^{-j}$. In $A$ itself, the components
corresponding to $H$ and $\sigma^2 H$, which are in fact double
cosets, are $K$ and $Ku^2$ respectively. Note that together $H' = H
\cup \s^2 H$ is a subgroup of $G$ containing $H$, which stabilizes
the subfield $K' = E^{\s^2}$. As shown in the example above, $H'$
corresponds to the sum of the components $K + Ku^2 = K[u^2] =
\operatorname{C}_A(K')$. The final component corresponds to the
double coset $H\s H$; over $F'$, this component decomposes as
$F'Ku+F'Ku^3$; but $u \not \in A$, so over $F$ we only have the
single component $Ku+Ku^2$. To summarize, the bimodule decomposition
of $A$ over $F$ is $A = K \oplus K u^2 \oplus (Ku+Ku^3)$.


\section{Characteristic coefficients and trace of powers}\label{sec:7}

Let $F$ be a field of arbitrary characteristic, and $A$ an $F$-algebra which is contained in $n \times n$ matrices over some extension of $F$.
Define $\rho_i\co A \to F$ to be the polynomial functions such that $(-1)^i\rho_i(a)$
is the coefficient of $t^{n-i}$ in the Cayley Hamilton polynomial $p_a(t)$ of $a$.

The well known Newton's identities are, for $k = 1,\dots,n$,
\begin{equation}\label{Newton}
k\rho_k(a) = \sum_{i=1}^{k}(-1)^{i-1}\rho_{k-i}(a)\tr(a^i).
\end{equation}

Fix an $F$-vector subspace $V \sub A$.
\begin{prop}\label{obvious}
Fix some $r \leq n$. Consider the conditions
\begin{enumerate}
\item[\TR] $\rho_k(a) = 0$ for $1 \leq k \leq r$ and every $a \in V$;
\item[\TT] $\tr(a^k) = 0$ for $1 \leq k \leq r$ and every $a \in V$.
\end{enumerate}
Then \TR $\implies$\TT\,  (and in characteristic zero, \TR $\Longleftrightarrow$ \TT).
\end{prop}
\begin{proof}
Since $\rho_0(a) = 1$, \eq{Newton} expresses $\tr(a^{r})$ as a linear combination of $r\rho_r(a)$ and the products $\rho_{r-k}(a)\tr(a^{k})$.
\end{proof}

Because of the presence of the integer $k$ at the left hand side of the above identity, we cannot obtain the converse statement, that if $\tr(a^i) = 0$ for $i\leq k$ then $\rho_k(a) = 0$ as well, unless we assume $\mychar F = 0$. Instead, we prove a characteristic-free multilinear version.
For any $k \leq r$, let
$$\tr(a_1,\dots,a_k) = \sum_{\eta \in S_\set{2,\dots,k}} \tr(a_1a_{\eta(2)}\cdots {a_{\eta(k)}}),$$
a sum of $(k-1)!$ traces of monomials, which is symmetric in the $k$ variables because the trace is invariant under cyclic shifts. Furthermore, let $\rho_k(a_1,\dots,a_k) \in F$ denote the coefficient of $t_1\cdots t_k$ in $\rho_k(a_1t_1+\cdots+a_k t_k)$, where the characteristic coefficient is taken in the extension $F[t_1,\dots,t_r] \tensor[F] A$. For $k = 1$ the newly defined $\tr(a_1)$ and $\rho_1(a_1)$ coincide with the usual definitions.

\begin{thm}\label{mult10}
Fix $r \leq n$. The following two conditions are equivalent:
\begin{enumerate}
\item[{\TRL}] $\rho_k(a_1,\ldots,a_k) = 0$ for all $1 \leq k \leq r$ and every $a_1,\dots,a_k \in V$;
\item[{\TTL}] $\tr(a_1,\ldots,a_k) = 0$ for all $1 \leq k \leq r$ and every $a_1,\dots,a_k \in V$.
\end{enumerate}
\end{thm}

\begin{rem}\label{infup}
\TR $\implies$\TRL\ when $F$ is infinite.
Indeed, consider $V[t] = V \otimes_F F[t_1,\ldots,t_k]$ and note that
$\rho_k(a) = 0$ for all $a \in V[t]$ since $F$ is infinite. Applying this
to $a = t_1a_1 + \ldots + t_ka_k$ we have that $\rho_k(a_1,\dots,a_k) = 0$. 
\end{rem}

In order to prove \Tref{mult10}, we need a version of \eq{Newton} where $k$ cancels. Consider the polynomial ring
$$R_0 = \Z[x_{i,j,s} \,|\, 1 \leq s \leq r,\  1 \leq i,j \leq n]$$
in $rn^2$ variables, and $R = R_0[t_1,\dots,t_r]$. In $\M[n](R)$, form the generic matrix
$X_s$ with $i,j$ entry $x_{i,j,s}$. Next form the generic sum
$X = \sum_s t_sX_s$. Let $T$ be the set $\{1,\ldots,r\}$ and let
$S \subseteq T$ (nonempty). Define $t^S$ to be the product of the $t_s$
where $s \in S$, so $t^T = t_1\cdots{t_r}$. Define $X_S = \sum_{s \in S} t_sX_s$. We adopt the following notation from \cite{Wilf}: $[t^S]p$ is the coefficient of the monomial $t^S$ in a polynomial $p \in R$.
Note that if $k = \card{S}$ then $[t^S]\rho_k(X_S) = [t^S]\rho_k(X)$.

For $S = \set{s_1,\dots,s_k}$, let $\tr_S = \tr(X_{s_1},\dots,X_{s_k})$, as defined above.
Clearly, $[t^S]\tr(X_S^k) = k \tr_S$, where we use the fact that
$\tr(x_1\cdots{x_k}$) remains invariant under cyclic shifts of the variables.
Once again we have that $k\tr_S = [t^S]\tr(X^k)$. Note that, as a special
case, $[t^T]\tr(X^r) = r \tr_T$.

We are interested in an identity for the multilinearization of
$\rho_r(X_T)$. To this end we multilinearize $\rho_{r-k}(x)\tr(x^k)$. Notice that for any two polynomials $p$ and $p'$,
$$[t^T](pp') = \sum_{S} [t^S]p \cdot [t^{T-S}]p',$$
where the sum is over all subsets $S \sub T$. In particular
\begin{eqnarray*}
{} [t^T]\rho_{r-k}(X)\tr(X^k) & = & \sum_S [t^S]\tr(X^k) \cdot [t^{T-S}]\rho_{r-k}(X) \\
& =  &  \sum_S k \tr_S \cdot [t^{T-S}]\rho_{r-k}(X_{T-S}),
\end{eqnarray*}
where the sum being over all $S \subset T$ with $|S| = k$, because $\tr(X^k)$ is homogeneous of degree $k$.

Substitute $X$ for $a$ in \eq{Newton}. Taking $t^T$ terms (and combining all the subsets $S$), we have
\begin{eqnarray}\label{rhorN}
[t^T]r\rho_r(X) & = & \sum_{k=1}^{r}(-1)^{k-1} [t^T]\rho_{r-k}(X)\tr(X^{k}) \nonumber \\
& = & \sum_{k=1}^{r}(-1)^{k-1} \sum_{\card{S} = k} k \tr_S \cdot [t^{T-S}]\rho_{r-k}(X_{T-S}) \nonumber \\
& = & \sum_{S\neq \emptyset} (-1)^{\card{S}-1} \card{S} \tr_S \cdot [t^{T-S}]\rho_{r-\card{S}}(X_{T-S}).
\end{eqnarray}

Let $\Delta = S_1 \cup \ldots \cup S_m$ be a partition of $T$ into nonempty parts.
Set $(-1)^{\Delta} = (-1)^{\sum_i (\card{S_i}-1)}$.
For this $\Delta$ we define $\tr_{\Delta} = \prod_i \tr_{S_i}$.
We claim:

\begin{thm}\label{NI+}
$\rho_r(X_1,\dots,X_r) = \sum_{\Delta} (-1)^{\Delta}\tr_{\Delta}$, where the sum is over all partitions of $T$.
\end{thm}

\begin{proof}
We prove this by induction on $r$. If $r = 1$ this just says
that $\rho_1(X_1) = \tr(X_1)$.
Assume the result for all $k < r$.
We start with the formula \eq{rhorN} for $\rho_r(X_T)$ and substitute
the expressions we have for $\rho_{r - |S|}(X_{T - S})$ by induction.
Note that if $\Delta'$ is a partition of $T - S$, and $\Delta$ is
$\Delta'$ with $S$ adjoined, then $(-1)^{\Delta} = (-1)^{\Delta'}(-1)^{|S|-1}$. We can thus compute:
\begin{eqnarray*}
{}[t^T]r \rho_r(X_T) &  = & \sum_{S\neq \emptyset} (-1)^{\card{S}-1} \card{S} \tr_S \cdot [t^{T-S}]\rho_{r-\card{S}}(X_{T-S}) \\
&  = & \sum_{S\neq \emptyset}
(-1)^{\card{S}-1}\card{S} \tr_S
\sum_{\Delta' \vdash T-S} (-1)^{\Delta'}   \tr_{\Delta'}\\
&  = & \sum_{S\neq \emptyset}
\sum_{\Delta' \vdash T-S} (-1)^{\card{S}-1}(-1)^{\Delta'} \card{S} \tr_S  \tr_{\Delta'}\\
&  = & \sum_{\Delta \vdash T} \sum_{S \in \Delta} (-1)^{\Delta} \card{S} \tr_S  \tr_{\Delta -S}\\
&  = & \sum_{\Delta \vdash T} (-1)^{\Delta} \tr_{\Delta} \cdot \sum_{S \in \Delta}  \card{S}  \\
&  = & r \sum_{\Delta \vdash
 T} (-1)^{\Delta} \tr_{\Delta},
\end{eqnarray*}
and the $r$'s cancel.
\end{proof}

\begin{proof}[Proof of \Tref{mult10}]
\TTL $\implies$\TRL: Fix a specialization of $\M[n](R_0)$ by sending $X_s$ to arbitrary elements of $V$. By assumption we have that $\tr_S = 0$ for every subset $S \sub T$, so $\tr_\Delta = 0$ for any partition $\Delta$ or $T$. By \Tref{NI+}, we obtain $\rho_r(X_1,\dots,X_r) = 0$. \TRL $\implies$\TTL: same argument by induction on $r$, since \Tref{NI+} presents $\tr(X_1,\dots,X_r)$ as a linear combination of $\rho_r(X_1,\dots,X_r)$ and products of values of the form $\tr_S$ for $\card{S} < r$.
\end{proof}

\begin{rem} In characteristic zero the conditions \TR,\TRL,\TT,\TTL\ are equivalent. More precisely, we have:
\begin{enumerate}
\item[1)] if $(r-1)!$ is invertible then \TTL $\implies$\TT;
\item[2)] if $r!$ is invertible then \TT $\implies$\TR\ and the four conditions coincide.
\end{enumerate}
Indeed, if $k$ is invertible then $\tr(a^k) = 0$ implies $\rho_k(a) = 0$ by Newton's formula. If $(k-1)!$ is invertible then $\tr(a_1,\dots,a_k) = 0$ implies $\tr(a^k) = 0$ by taking $a_1=\cdots=a_k = a$. The claim follows by ranging over all $1 \leq k \leq r$.
\end{rem}

\begin{exmpl}
We show that the conditions \TT\ and \TTL\ are independent when $\mychar F = 2$ and $r = 3$. We take $V$ to be a space of diagonal matrices in $\M[6](F)$.
Notice that $\tr(a_1,a_2) = \tr(a_1a_2)$ and $\tr(a_1,a_2,a_3) = \tr(a_1a_2a_3)+\tr(a_1a_3a_2)$. Since elements of $V$ commute, $\tr(a_1,a_2,a_3) = 0$ automatically.

\TTL $\ \ \not\!\!\! \implies$\TT:
Fix some $\alpha \neq 0,1$ in $F$, and let $\alpha' = \alpha+1$. Let $V$ be spanned by the matrices with diagonals $(1,0,0,\, 0,\alpha,\alpha')$,  $(0,1,0,\,0,\alpha',\alpha)$, and  $(0,0,1,\,1,0,0)$. Then $\tr(a_1) = 0$ and $\tr(a_1a_2) = 0$ for every $a_1,a_2 \in V$, so \TTL\ holds. However $1+\alpha^3+\alpha'^3 = \alpha\alpha' \neq 0$, so \TT\ fails.

\TT $\ \ \not\!\!\! \implies$\TTL:
For the converse, assume $\rho \in F$ is a primitive third root of unity. Let $V$ be spanned by the matrices with diagonals $(1,0,0,\,1,\rho,\rho)$,  $(0,1,0,\,\rho,1,\rho)$ and $(0,0,1,\,\rho,\rho,1)$. Then $\tr(a_1) = 0$ for every $a_1 \in V$, which implies $\tr(a_1^2) = \tr(a_1)^2 = 0$; $\tr(a_1^3) = 0$ holds by computation, 
which confirms \TT. However the identity $\tr(a_1a_2) = 0$ fails, and so does \TTL.
\end{exmpl}


\section{Criteria for $KaK$ to be a Kummer space}

As assumed throughout this paper,
$K \subset A$ is a maximal separable subfield, and  $G\supset H$
are the Galois groups of $L/F$ and $L/K$ respectively, where $L/F$
is the Galois closure of $K/F$. We set $n = [K\,{:}\,F]$. An $F$-vector space $V \sub A$ is {\bf{Kummer}} if for every $v\in V$, $\rho_i(v) = 0$ for $i \in \set{1,\dots,n-1}$, so in particular $v^n \in F$.


If $aKa^{-1} = K$, and $a$ induces an automorphism $\s \ne 1$ on
$K$ which generates the Galois group of $K/F$, then $a^n \in
C_A(K)^\s= K^\s = F$ but $a^d \not \in F$ for a proper divisor $d \,|\, n$, implying $\rho_i(a) = 0$ for each $1 \le i \le n-1$. Since we may replace $a$ by any element of $Ka$, it then follows that $Ka$ is a Kummer subspace. We seek to prove a converse.

We first consider the question when is $V = KaK$ a Kummer subspace (i.e. \TR\ of \Pref{obvious} for $r = n$).
Let ${\mathcal H}$ be be the $H$-$H$ biset of $G$ associated to $KaK$.
For convenience, we write the coset $gH$ as $g$ when it appears in a subscript.

\begin{prop}\label{P41}
The following are equivalent.

a) For all $x \in KaK$ and all $1 \leq k \leq r$, $\rho_k(x) = 0$.

b) For all $1 \leq k \leq r$ there are no $g_1,\ldots,g_k \in {\mathcal H}$ such that
$g_1g_2\cdots{g_k} = 1$.\end{prop}

\begin{proof} Assume a).
By way of contradiction, suppose such
$g_1,\ldots,g_k$ exist. By induction we can assume that b) holds
for all $k < r$. Since $F$
is infinite, it is also true that $\tr(x^k) = 0$ for all $x \in \bar
Ka\bar K$, where, as always $\bar{K} = L \tensor K$ and $L$ is the Galois closure of $K/F$. Choose any $\tau \in G$ and set $v_i =
e_{{\tau}g_1g_2\cdots{g_{i-1}},{\tau}g_1g_2\cdots{g_{i-1}}g_i}$. Of
course
$({\tau}g_1g_2\cdots{g_{i-1}})^{-1}({\tau}g_1g_2\cdots g_{i-1}g_i) =
g_i \in {\mathcal H}$ so $v_i \in \bar Ka\bar K$. Also, $v_1 =
e_{\tau,\tau{g_1}}$ and $v_k = e_{{\tau}g_1\cdots{g_{k-1}},\tau}$.
Then $v_1\cdots{v_k} = e_{\tau,\tau}$. Since $g_i\cdots g_j = 1$ for $i \leq j$ can only hold when $i = 1$ and $j = k$ by assumption, the elements $\tau g_1\cdots g_{i-1}$ ($i = 1,\dots,k$) are distinct. Therefore, the only non-zero product $v_1v_{\eta(2)}{\cdots}v_{\eta(k)} \not = 0$, for a permutation $\eta \in S_\set{2,\dots,k}$, is obtained when $\eta$ is the identity.
Thus $\tr(v_1,\ldots,v_k) = 1$, contrary to the condition \TTL\ of \Tref{mult10}; which follows from \TR\ by that theorem and \Rref{infup}, a contradiction.

Conversely, assume b). Let $R = \Z[x_{g,g'}\,|\, g,g' \in G/H]$
be the polynomial ring in $n^2r$ variables. We write $X$ to be the matrix
with $x_{g,g'}$ in the $(g,g')$ entry when $g^{-1}g' \in {\mathcal H}$
and 0 otherwise. In effect, $X$ is a generic element of $\bar K{a}\bar K$ but over $\Z$.
It suffices to show $\rho_k(X) = 0$ because $X$ specializes to all
elements of $\bar K{a}\bar K$. Thus we reduce to the case $F$ has characteristic 0.
(The reader may object that there is no $K$ etc. over $\Z$. There are two answers,
one being that all the essential properties of $\bar K{a}\bar K$ are present
over $\Z$, or alternatively we can embed $\M[n](\Z)$ into some $\bar A$ so that the
matrix idempotents are preserved.)

In characteristic 0 it suffices to prove that $\tr(x^k) = 0$ for all $1 \leq k \leq r$.
By induction it suffices to prove $\tr(x^r) = 0$.
Write
$$x = \sum_{\tau\in G, g\in {\mathcal H}} d_{\tau,g}e_{\tau,\tau g}$$
where $d_{\tau,g} \in \bar{K}$. Then expanding
$x^r$ there is no term of the form $\bar{K} e_{\tau,\tau}$ by the assumption on $\mathcal{H}$.
\end{proof}

\begin{rem}
By the above we see that $\tr(v)=0$ for all $v\in V$ if and only if
$H\not \subseteq \mathcal{H}$. Thus zeroing out the characteristic
coefficients corresponds to avoiding having $H$ in the powers of the
associated $H$-biset $\mathcal{H}$.

For example,\begin{enumerate}
 \item Assume that $G=\langle \sigma \rangle$ is cyclic, namely $H=\{1\}$. Then the bimodule  $V$
 corresponding to the $H$-biset $\{\sigma\}$ satisfies, $\rho_k(v)=0$ for all $1\leq k\leq n-1$; indeed,
 $\{\sigma\}^k\cap H=\emptyset$ for all the relevant cases.
     \item Assume that $G=\langle \sigma \rangle\rtimes \langle \tau \rangle \cong C_n \rtimes C_2$ is dihedral,
     namely $H=\langle \tau \rangle$.
     Then the bimodule $V$ corresponding to the $H$-biset $H\sigma H$ satisfies  $\rho_k(v)=0$ for all odd $k$,
     $1\leq k\leq n-1$;
      indeed, $\{\sigma\}^k\cap H=\emptyset$ for all the relevant cases. This was used in \cite{RS}  to prove that dihedral algebras of odd degree are cyclic.
 \end{enumerate}

\end{rem}

We are interested in saying something about elements in $K a$, where again we may pass to $\bar{K}a$. Notice that in $\bar K a$ we can choose  elements of the form $x = \sum_{\tau^{-1} \sigma \in \mathcal H} x_{\tau}e_{\tau,\sigma}$ for a fixed $\sigma$. Fixing $\tau$ with $\tau^{-1}\sigma \in {\mathcal H}$, we can always choose
$x$ with $x_{\tau} = 1$ because $e_{\tau,\sigma} \in \bar Ka\bar K$
and so $e_{\tau,\sigma} = e_{\tau,\tau}x$ for $x \in a\bar K$.
We claim:
\begin{thm}\label{Kum5}
Suppose that $\rho_k(x) = 0$ for all $1 \leq k \leq r$ and all $x
\in Ka$. Then the same is true for all $x \in KaK$.
\end{thm}

\begin{proof}
We proceed by induction on $r$. If $\rho_1(ak) = \tr(aK) = 0$ for all $k
\in K$ then $\tr(k'ak) = \tr(akk') = 0$ for all $k',k \in K$ and any
element of $KaK$ is a sum of $k'ak$'s.

For $r > 1$ the result is harder but we know by induction that
$\rho_i(x) = 0$ for all $i < r$ and all $x \in KaK$. In particular,
for all $i < r$ there are no $g_1,\ldots, g_i \in {\mathcal H}$ with
$g_1\ldots{g_i} = 1$ by \Pref{P41}.

Again by \Pref{P41}, it suffices to show that there are no $g_1,\ldots,g_r$ with $g_1\cdots{g_r} = 1$. Assume otherwise.
Choose $\tau$ arbitrary, define $\tau_0 = \tau$ and set $\tau_i =
\tau{g_1}\ldots{g_i}$ for $i = 1,\ldots,r$. Choose $v_i =
\sum_{\tau'} x_{i,\tau'}e_{\tau',\tau_i} \in a\bar
Ke_{\tau_i,\tau_i}$ such that $x_{i,\tau_{i-1}}$, the coefficient of
$e_{\tau_{i-1},\tau_i}$, equals 1. Then $v_1\cdots{v_r}$ has a unique diagonal element namely $e_{\tau,\tau}$. By Proposition \ref{mult10}, there is a $1
\not= \eta \in S_{r-1}$ such that
$v_1v_{\eta(2)}\ldots{v_{\eta(r)}}$ has a nonzero
diagonal component $ye_{\tau,\tau}$. If $1 \not= j = \eta(1)$, then
$j > 1$ and in $v_j = \sum_{\tau'} x_{j,\tau'}e_{\tau',\tau_j}$ we
must have $x_{j,\tau} \not= 0$. This implies $\tau^{-1}\tau_j =
g_1\ldots{g_j} \in {\mathcal H}$ which implies
$(g_1\ldots{g_j})g_{j+1}\ldots{g_r} = 1$ which contradicts the
induction assumption. Thus $\eta(1) = 1$. If we redefine $j =
\eta(2)$ and assume that $j > 2$, then arguing similarly we have
$x_{\tau_1,\tau_j} \not= 0$ and so $\tau_2\ldots{\tau_j} \in
{\mathcal H}$. This again contradicts the induction assumption and
so, by an inner induction that proceeds now in the obvious way, we
have $\eta$ is the identity, proving the theorem.~\end{proof}

\begin{cor}\label{bisp}
 $Ka$ is Kummer iff   $KaK$ is Kummer.
\end{cor}

\begin{thm}\label{Kum1}
Suppose that $Ka$ is Kummer. Then $H = \set{1}$, $G$ is cyclic of
order $n$ with generator $g$, and $KaK = Ku$ for invertible $u\in
KaK$ satisfying $uku^{-1} = g(k)$ for all $k \in K$. If $a$ is
invertible we may assume $u = a$.
\end{thm}

\begin{proof}
By Corollary~\ref{bisp}, $KaK$ is Kummer. Let ${\mathcal H}$ be the
$H$-biset associated to $KaK$. By \Pref{P41} there are no $g_i \in
{\mathcal H}$ such that $g_1g_2{\ldots}g_s = 1$ for any $1 \leq s <
n$. Equivalently, $H$ is not contained in ${\mathcal H}^s$ for any
$1 \leq s < n$. Fix some $g \in {\mathcal H}$. If $g^iH = g^jH$ for
$i < j$, then $g^{j-i} \in H$ and so $H \subset {\mathcal H}^{j-i}$
implying $j - i \geq n$. Thus $H, gH, \ldots, g^{n-1}H$ are all
distinct. This implies that their union is $G$ and hence $g^nH = H$
or $g^n \in H$, and $g^s \notin H$ for $1 \leq s < n$.

We show that $\mathcal H = gH$. Indeed, assume $g^sH \sub \mathcal H$; then $g^n = g^{n-s}g^s \in \mathcal{H}^{n-s}\mathcal{H} = \mathcal{H}^{n-s+1}$, and $H \sub \mathcal{H}^{n-s+1}$, implying that $s = 1$. But now $gH$ is an $H$-biset, so $Hg = gH$ and $H \lhd G$ since $\sg{H,g} = G$.

Since $L/F$ was the Galois closure of $K/F$, this
implies $H = (1)$ and now clearly $G$ is cyclic of order $n$, and
$KaK$ is associated to the double coset $HgH = \set{g}$ for $g$ a generator of
$G$. But this implies $KaK = uK$ where $uku^{-1} = g(k)$ for all
$k \in K$. \end{proof}


\section{Algebras of prime degree}

Assume that the degree of $A$ over $F$ is a prime $p$. We observe
that the dimensions of the irreducible components, other than $K$,
are all equal.

\begin{prop}\label{options}
Suppose that $A/F$ has prime degree $p$. There are two
possibilities:
\begin{enumerate}
\item For every $a \not \in K$, $\dim_K(KaK) = p-1$. In this case $G$ is doubly transitive.
\item All irreducible sub-$K$-$K$-bimodules of $A$ other than $K$ have the same dimension $r$; and $L/F$ has Galois group $C_p \rtimes C_r$ where $C_r$ acts faithfully on $C_p$, so $r$ (strictly) divides $p-1$.
\end{enumerate}
\end{prop}
\begin{proof}
If (1) fails, the group $G$ is transitive, but not doubly transitive
by \Rref{dt}. By a theorem of Burnside~\cite[Theorem~I.7.3]{Pass}, $G$ then is
solvable and hence of the form $C_p \rtimes C_r$. $H$ is now a
conjugate of $C_r$ and its double cosets correspond to the orbits of
$C_r$ on~$C_p$.
\end{proof}

Notice that the case $r = 1$ is when $K$ is cyclic. It is interesting to consider the ``next best'' case where (some) $KaK$ has dimension exactly~2 over $K$.
If $K/F$ is not cyclic, then there is a double coset $HgH$ which has
order $2\card{H}$. By the proposition this forces $L/F$ to have dihedral Galois group -- which means $A$ is cyclic by \cite{RS} (although not with respect to $K$).

In the situation (1), $A = K \oplus KaK$ for some $a \not \in K$, and $(KaK)^2 = A$. Let us look more closely at the situation (2) in the above proposition.
That is, assume that $n=p$ is prime, and $G$ is transitive but not
doubly transitive. Again by Burnside's theorem, $G$ is contained in the
affine group of the field $\F_p$, and contains the element $\s(i)
= i+1$. In the earlier notation, we may assume $H$ is the
stabilizer of $0$, which is cyclic of some order $r$ strictly dividing $p-1$. We present
$$G = C_p \rtimes C_r = \sg{\s, \tau \,|\, \s^p = \tau^r = 1, \tau\s\tau^{-1} = \sigma^t};$$
thus $H = \sg{\tau}$, which we identify with the subgroup $\sg{t}
\sub \mul{\F_p}$ where we fix an element $t \in \mul{\F_p}$ of
order $r$.

For $c \in \F_p$, let us denote $$\s^{cH} = \set{\tau' \s^c
\tau'^{-1} \,|\, \tau' \in H} = \set{\s^{ct^i} \,|\, i =
0,\dots,r-1};$$ thus $\s^{cH}H = H \s^{cH}$. Any double coset of
$H$ in $G$ has the form $H \s^c \tau^d H = H \s^c H = \set{\tau^i
\s^c \tau^{-i} \,|\, \tau^i \in H} H = \s^{cH}H$. The double
cosets correspond to the orbits of $\F_p$ under the action of $H =
\sg{t}$ by multiplication.
The inverse of this double coset is $(\s^{cH}H)^{-1} = H\s^{-cH} =
\s^{-cH}H$, and the product (in $G$) of two double cosets
$\s^{cH}H, \s^{c'H}H$ is
 $ \s^{cH+c'H}H$, which is a union of all the
$\s^{c''H}H$ for which $c''H \sub cH+c'H$. Thus, the semiring of
double cosets with union and multiplication is isomorphic to the semiring of subsets of the quotient group $\mul{\F_p}/\sg{t}$ with union and addition of subsets.

We may identify $\bar A = L \tensor A$ with $\M[p](F)$, with $e_{ii}$
the idempotent corresponding to $\s^i H$. The irreducible
bimodule corresponding to a double coset $\s^{cH}H$ is~$K\sum_{i-j
\in cH} e_{ij} K$.


\printindex

\end{document}